\documentclass[11p,reqno]{amsart}
\usepackage{amsmath}
\usepackage{amsfonts}
\usepackage{amssymb}
\usepackage{graphicx}
\usepackage{color}
\newtheorem{theorem}{Theorem}[section]
\newtheorem*{theorem*}{Theorem}
\newtheorem{lemma}{Lemma}[section]
\newtheorem{corollary}[theorem]{Corollary}
\newtheorem{proposition}{Proposition}[section]

\def\i{\sqrt{-1}}
\def\Ric{\text{Ric}}

\def\i{\sqrt {-1}}

\def\aint{\frac{\ \ }{\ \ }{\hskip -0.4cm}\int}

\def\Ric{\operatorname{Ric}}

\numberwithin{equation}{section}

\begin{document}
	\title[positivity of $k$-scalar curvature]{Positivity and Kodaira embedding theorem}

\author{Lei Ni}\thanks{The research of LN is partially supported by NSF grant DMS-1401500 and  the  ``Capacity Building for Sci-Tech Innovation-Fundamental Research Funds".}
\address{Lei Ni. Department of Mathematics, University of California, San Diego, La Jolla, CA 92093, USA}
\email{leni@ucsd.edu}

\author{Fangyang Zheng} \thanks{The research of FZ is partially supported by a Simons Collaboration Grant 355557.}
\address{Fangyang Zheng. School of Mathematical Sciences, Chongqing Normal University, Chongqing 401331, China}
\email{{franciszheng@yahoo.com}}


\begin{abstract}
Kodaira embedding theorem provides an effective characterization of projectivity of a K\"ahler manifold in terms the second cohomology.
 Recently  X. Yang \cite{XYang} proved that any compact K\"ahler manifold with positive holomorphic sectional curvature must be projective. This gives a metric criterion of the projectivity in terms of its curvature. In this note, we prove that any compact K\"ahler manifold with positive 2nd scalar curvature (which is the average of holomorphic sectional curvature over $2$-dimensional subspaces of the tangent space) must be projective. In view of generic 2-tori being non-Abelian, this new curvature characterization  is sharp in certain sense.

\end{abstract}

\maketitle

\section{Introduction}

Let $(M^m, g)$ be a K\"ahler manifold with complex dimension $m$. For $x\in M$, denote by $T_x'M$ the holomorphic
tangent space at $x$. Let $R$ denote the curvature tensor. For $X\in T_x'M$ let $H(X)=R(X, \bar{X}, X, \bar{X})/|X|^4$ be
the holomorphic sectional curvature. Here $|X|^2=\langle X, \bar{X}\rangle$, and we extended the Riemannian product $\langle \cdot, \cdot\rangle$  and the curvature tensor $R$  linearly over ${\mathbb C}$, following the convention of \cite{NZ}. We say that $(M, g)$ has
positive holomorphic sectional curvature, if $H(X)>0$ for any $x\in M$ and any $0\neq X\in T_x'M $. It was known that compact manifolds with positive holomorphic sectional curvature must be simply connected \cite{Tsu}. A three circle property was established for noncompact complete K\"ahler manifolds with nonnegative holomorphic sectional curvature \cite{Liu}. On the other hand it was known that such metric may not even have positive Ricci curvature \cite{Hitchin}.

The following result was proved by X. Yang in \cite{XYang} recently, which  answers affirmatively a question in \cite{YauP}.

\medskip

{\it If the compact K\"ahler manifold $M$ has positive holomorphic sectional curvature, then $M$ is
projective. Namely $M$ can be embedded into a complex projective space via a holomorphic map.}

\medskip

The key step is to show that the Hodge number $h^{2, 0}=0$. Then a well-known result of Kodaira  (cf. Chapter 3,
Theorem 8.3 of \cite{MK}) implies the projectiveness.

The purpose of this paper is to prove a generalization of the above result of Yang. First of all we introduce some
notations after recalling  a lemma of Berger.

\begin{lemma}\label{lemma:11}
If $S(p)=\sum_{i, j=1}^m R(E_i, \overline{E}_i, E_j, \overline{E}_j)$, where $\{E_i\}$ is a unitary basis of $T_p'M$,
denotes the scalar curvature of $M$, then
\begin{equation}\label{eq:11}
2S(p)= \frac{m(m+1)}{Vol(\mathbb{S}^{2m-1})}\int_{|Z|=1, Z\in T'_pM} H(Z)\, d\theta(Z).
\end{equation}
\end{lemma}
\begin{proof} Direct calculations shows that
$$
\frac{1}{Vol(\mathbb{S}^{2m-1})}\int_{\mathbb{S}^{2m-1}} |z_i|^4 =\frac{2}{m(m+1)}, \ \
\frac{1}{Vol(\mathbb{S}^{2m-1})}\int_{\mathbb{S}^{2m-1}} |z_i|^2|z_j|^2 =\frac{1}{m(m+1)}
$$
for each $i$ and each $i\neq j$. Equation (\ref{eq:11}) then follows by expanding $H(Z)$ in terms of $Z=\sum_i z_i E_{i}$, and  the above
formulae.
\end{proof}
For any integer $k$ with $1\le k \le m$ and any $k$-dimensional subspace $\Sigma\subset T_x'M$, one can defined the $k$-scalar
curvature as
$$
S_k(x, \Sigma)=\frac{k(k+1)}{2Vol(\mathbb{S}^{2k-1})}\int_{|Z|=1, Z\in \Sigma} H(Z)\, d\theta(Z).
$$
By the above Berger's lemma, $\{S_k(x, \Sigma)\}$ interpolate between the holomorphic sectional curvature, which is
$S_1(x, \{X\})$,  and scalar curvature, which is $S_m(x, T_xM)$.

We say that $(M, g)$ has positive 2nd-scalar curvature if $S_2(x, \Sigma)>0$ for any $x$ and any two complex plane
$\Sigma$.

Clearly, the positivity of the holomorphic sectional curvature implies the positivity of the 2nd-scalar curvature, and the positivity of $S_k$ implies the positivity of $S_l$ if $k\leq l$. We
shall prove the following generalization of above mentioned result of Yang.

\begin{theorem}\label{thm:11} Any compact K\"ahler manifold $M^m$ with positive 2nd-scalar curvature must be projective. In
fact $h^{2, 0}(M)=0$.
\end{theorem}

Recall that a projective manifold $M$ is said to be {\em rationally connected}, if any two generic points in it can be connected by a chain of rational curves. By the work of \cite{KMM}, any projective manifold $M$ admits a rational map $f: M \dashrightarrow Z$ onto a projective manifold $Z$ such that any generic fiber is  rationally connected, and for any very general point (meaning away from a countable union of proper subvarieties) $z\in Z$, any rational curve in $M$ which intersects the fiber $f^{-1}(z)$ must be contained in that fiber. Such a map is called a {\em maximal rationally connected fibration} for $M$, or {\em MRC fibration} for short. It is unique up to birational equivalence. The dimension of the fiber of a MRC fibration of $M$ is called the {\em rational dimension} of $M$, denoted by $rd(M)$.

Heier and Wong (Theorem 1.7 of \cite{HeierWong}) proved that any projective manifold $M^m$ with $S_k>0$ satisfies $rd(M)\geq m-(k-1)$. So as  a corollary of their result and Theorem 1.1 above, we have the following consequence.

\begin{corollary}\label{coro:11} Let $M^m$ be a compact K\"ahler manifold with positive 2nd scalar curvature. Then $rd(M)\geq m-1$, namely, either $M$ is rationally connected, or there is a rational map $f: M\dashrightarrow C$ from $M$  onto a curve $C$ of positive genus, such that over the complement of a finite subset of C, f is a holomorphic submersion with compact, smooth fibers, each fiber is a rationally connected manifold.
\end{corollary}

Note that the intrinsic criterion of the 2nd scalar curvature can be used to imply that all compact Riemann surfaces (by taking a product with a very positive $\mathbb{P}^1$) are projective while Yang's result (under the positivity of holomorphic sectional curvature) can only be applied to $\mathbb{P}^1$. In the mean time a generic two complex dimensional tori is not algebraic. Hence the projectivity can NOT be possibly implied by the positivity of $S_k$ with $k\ge 3$ (taking the product of a non-algebraic tori of complex dimension 2 with a very positive $\mathbb{P}^1$ one can endow a K\"ahler metric with $S_k>0$ for $k\ge 3$ on such a non-algebraic manifold). In view of these examples our result is sharp in some sense. Moreover the positivity of $S_2$ is stable (namely a open condition) under the holomorphic  deformation of the complex manifolds (along with the smoothly deformation of the K\"ahler metrics specified by Kodaria-Spencer \cite{MK}). Hence our result provides a condition invariant under the small deformation of holomorphic structure. On the other hand, there are celebrated examples of Voisin \cite{Voi} of K\"ahler manifolds of complex dimension four and above, which can not be deformed into an algebraic one via a complex holomorphic deformation, and the wildly open Kodaira's problem in complex dimension three asking whether or not a K\"ahler threefold can be deformed into a projective manifold.

It is well known that $h^{m, 0}=0$ if $(M^m, g)$ has positive scalar curvature. The traditional Bochner formula also implies the vanishing of $h^{p, 0}=0$ for $k\leq p \leq  m$ if the Ricci curvature of $(M^m, g)$ is $k$-positive, namely the sum of the smallest $k$ eigenvalues of the Ricci tensor is positive (cf. \cite{Ko2}). The following result also holds.

\begin{theorem} \label{thm:12} Let $(M^m, g)$ be a compact K\"ahler manifold. If the $k$-th scalar curvature is positive, then $h^{p, 0}=0$ for any $k\leq p\leq m$.
\end{theorem}

It turns out that the original argument of proving the above result contains an error. However it can be proved using a maximum principle consideration via the co-mass (a $L^\infty$-norm) of differential forms. Please see \cite{Ni-19} Proposition 4.2 and Corollary 4.3 for details.

As a counterpart to Theorem 1.7 of \cite{HeierWong}, one can ask the question that, for a given projective K\"ahler manifold $M^m$ with $S_k<0$, what is the maximal possible rational dimension?   A naive conjecture which mimics the Heier-Wong's Theorem would be: $ S_k<0 \ \ \Longrightarrow \ \ rd(M)<k$.
For $k=m$, the conjecture says that having negative scalar curvature would imply the manifold cannot be rationally connected. This is still unknown even for $m=2$ as far as we know. (Masataka Iwai \cite{Iwai} shared an example of complex surface with a {\it Hermitian metric} of negative scalar curvature which is rationally connected.) On the other hand, $S_m<0$ (or just the integral of the scalar curvature being negative) does imply that $H^0(M, K_M^{-\otimes \ell})=0$ for any $\ell >0$, where $K_M^{-1}$ is  the anti-canonical line bundle, so $M$ cannot be a Fano manifold when $S_k<0$ for any $k$. Note also  that a recent result in \cite{Ni-ijm} (cf. Theorem 5.1) implies that any holomorphic map from $\mathbb{P}^2$ or a two dimensional tori into a K\"ahler manifold $M^m$ (not necessarily compact) with $S_2<0$  is either constant or of rank one.

We should mention that there is also a recent work of Wu and Yau \cite{WuYau} on the ampleness of the canonical line bundle assuming the holomorphic sectional curvature being negative, which is another perfect example of getting algebraic geometric consequence in terms of the metric property via the holomorphic sectional curvature.

Generally speaking, we think it is interesting to obtain algebraic geometric characterizations of  condition $S_k>0$ or $S_k<0$, as well as the conditions $Ric^\perp>0$, $Ric^\perp <0$ studied recently in \cite{NZ} by the authors, where an complementary metric criterion of the projectivity was given in terms of $Ric^\perp_2>0$. A complete classification result for threefolds and a partial classification of fourfolds have been obtained (cf. \cite{NZ-19}) for K\"ahler manifolds with $Ric^\perp>0$.
The estimates developed in the proof of this paper have also been useful  \cite{Ni-19} in proving the rational-connectedness of K\"ahler manifolds with $Ric_k>0$. We refer the interested readers to \cite{Ni-19} for these and other notions of curvature positivities as well as many related results and questions.

\section{The projectivity of $M$ with positive $S_2$}

Here we adopt the argument of \cite{NZ} to show that the dimension of $\mathcal{H}^{2, 0}(M)$, the space harmonic $(2, 0$-forms, $h^{2, 0}(M)=0$. Then Theorem 8.3 of \cite{MK} implies that $M$ is projective.

First recall the formula below (cf. Ch III, Proposition 1.5 of \cite{Ko2}, as well as Proposition 2.1 of
\cite{Ni-JDG}).

\begin{lemma}\label{lemma:21} Let $s$ be a global holomorphic $p$-form on $M^m$ which locally is expressed as  $s=\frac{1}{p!}\sum_{I_p} a_{I_p}dz^{i_1}\wedge \cdots \wedge dz^{i_p}$ ,  where $I_p=(i_1, \cdots, i_p)$. Then
  $$\partial \overline{\partial } \,|s|^2 = \langle \nabla s , \overline{\nabla s} \rangle  - \widetilde{R}(s,
  \overline{s}, \cdot , \cdot ) $$
where $\widetilde{R}$ stands for the curvature of the Hermitian bundle $\bigwedge^p\Omega$, where $\Omega=(T'M)^*$ is
the holomorphic cotangent bundle of $M$. The metric on $\bigwedge^p\Omega$ is derived from the metric of $M^m$. Then for any unitary frame $\{dz^j\}$,
\begin{equation}\label{eq:20}
\langle \sqrt{-1}\partial\bar{\partial} |s|^2, \frac{1}{\sqrt{-1}}v\wedge \bar{v}\rangle =\langle \nabla_v s,
\bar{\nabla}_{\bar{v}} \bar{s}\rangle +\frac{1}{p!}\sum_{I_p} \sum_{k=1}^p  \sum_{l=1}^m \langle R_{v\bar{v}i_k \bar{l}}a_{I_p}, \overline{a_{i_1\cdots(l)_k\cdots i_p}}\rangle.
\end{equation}
 Given any $x_0$ and $v\in T_{x_0}'M$, there exists a unitary frame $\{dz^i\}$ at $x_0$, which may depends on $v$, such that
\begin{equation}\label{eq:21}
\langle \sqrt{-1}\partial\bar{\partial} |s|^2, \frac{1}{\sqrt{-1}}v\wedge \bar{v}\rangle =\langle \nabla_v s,
\bar{\nabla}_{\bar{v}} \bar{s}\rangle +\frac{1}{p!}\sum_{I_p} \sum_{k=1}^p R_{v\bar{v}i_k \bar{i}_k}|a_{I_p}|^2.
\end{equation}
\end{lemma}

We prove the result by contradiction argument. Assume that $\mathcal{H}^{2, 0}(M)\ne \{0\}.$
 Let $s\in \mathcal{H}^{2, 0}(M)$ be a nonzero harmonic form. It is well-known that it is holomorphic. Write $s=\sum_{i,j} f_{ij}\varphi_i\wedge \varphi_j$ under any unitary coframe $\{\varphi_j\}$ which is dual to a local unitary tangent frame $\{\frac{\partial}{\partial z_j}\}$. The $m\times m$ matrix  $A= (f_{ij})$ is  skew-symmetric.  Note that there exists a normal form for any holomorphic $(2, 0)$-form $s$ at a given point $x_0$, (cf. Corollary 4.4.19 of \cite{HJ}). More precisely, given any skew-symmetric matrix $A$, there exists a unitary matrix $U$ such that $\ ^t\!U AU$ is in the block diagonal form where each non-zero diagonal block is a constant multiple of $F$, with
$$ F=\left[ \begin{array}{cc} 0 & 1 \\ -1 & 0 \end{array} \right]. $$
In other words, we can choose a unitary coframe $\varphi$ at $x_0$ such that
$$ s = \lambda_1 \varphi_1\wedge \varphi_2 + \lambda_2 \varphi_3 \wedge \varphi_4 + \cdots + \lambda_k \varphi_{2k-1}\wedge \varphi_{2k}, $$
where $k$ is a positive integer and each $\lambda_i\neq 0$ for $1\le i\le k$.

Suppose $k$ is the unique positive integer such that  $s^{k+1}=0$  while $s^k$ is not identically zero, and consider the holomorphic $2k$-form $\sigma=s^k$. By the argument on p151 of \cite{NZ}, we know that $\sigma=\lambda \varphi_1\wedge \cdots\wedge \varphi_{2k}\ne 0$. Now we apply Lemma 2.1 to $\sigma$ at the point $x_0$, where $|\sigma|^2$ attains its maximum and have that
$$
0\ge \langle \sqrt{-1}\partial\bar{\partial} |\sigma |^2, \frac{1}{\sqrt{-1}}v\wedge \bar{v}\rangle \ge |\lambda|^2 \sum_{i=1}^{2k} R_{i\bar{i}v\bar{v}},
$$
for any $v$. Taking $v=\frac{\partial}{\partial z^i}$, the dual of $\varphi_i$ at $x_0$ and sum them over we have that  at $x_0$
\begin{equation}\label{eq:contra}
\sum_{i, j=1}^{2k} R_{i\bar{i}j\bar{j}}\le 0
\end{equation}
On the other hand, it is easy to see that $S_2>0$ implies that
$S_{2k}>0$.  This is a contradiction to (\ref{eq:contra}). Hence there is no nonzero $s\in \mathcal{H}^{2, 0}(M)$.

In \cite{Ni-19}, via a different technique,  the result has been extended to K\"ahler manifolds with so-called RC-2 positivity, namely for any two unitary vectors  $\{E_1, E_2\}$, there exists $v$ such that $R(v, \bar{v}, E_1, \overline{E}_1)+R(v, \bar{v}, E_2, \overline{E}_2)>0$.

\section{Some related estimates}

Let $\Sigma$ be a $2$-plane where $S_2(x_0, \Sigma)=\inf_{\Sigma'} S_2(x_0, \Sigma')$, integrating the Bochner formula of Lemma \ref{lemma:21} for $v\in \mathbb{S}^3\subset \Sigma$, we have
\begin{equation}\label{eq:22}
\aint \partial_{v}\bar{\partial}_{\bar{v}} |s|^2=\aint \langle \nabla_v s, \bar{\nabla}_{\bar{v}} \bar{s}\rangle +\sum_{i=1}^k |a_{ij}|^2
\aint (R_{v\bar{v} i\,\overline{i}}+  R_{v\bar{v}j\, \overline{j}} ).
\end{equation}
Here $\aint f(Z)$ denote the average of the integral of the function  $f$ over $\mathbb{S}^3\subset \Sigma$.
We also have choose a unitary frame of $T_{x_0}$ such that $\aint R(v, \bar{v}, \cdot, \overline{(\cdot)})$ is diagonalized and $s$ is a holomorphic $2$-form given by  $s=\sum_{i\ne j}a_{ij}dz^i\wedge dz^j$.

A possible alternate approach to Theorem \ref{thm:11} is to apply the maximum principle at $x_0$, where $|s|^2$ attains its maximum. In
view of the compactness of the Grassmannians  we can also find a complex two plane $\Sigma$ in $T_{x_0}'M$ such that $S_2(x_0,
\Sigma)=\inf_{\Sigma'}S_2(x_0, \Sigma')>0$. We prove the following estimates, some of which were used in establishing the rational connectedness of algebraic manifolds under the $\Ric_k>0$ condition in \cite{Ni-19}.

\begin{proposition}\label{prop:21} For  any $E\in \Sigma$ and $E'\perp \Sigma$ with $|E|=|E'|=1$ and a two-dimensional plane $\Sigma'\subset T_p'M$ with $\Sigma'\ne \Sigma$ and a unitary frame $\{v_1, v_2\}$ of $\Sigma'$, we have that
\begin{eqnarray}
\aint R(E,\bar{E}', Z, \bar{Z})d\theta(Z) & = & \aint R(E', \bar{E}, Z, \bar{Z})d\theta(Z)\ = \ 0, \label{eq:1st} \\
\aint R(v_1, \overline{v}_1, Z, \bar{Z})+R(v_2, \overline{v}_2, Z, \bar{Z})\, d\theta(Z) &\ge& \frac{1}{3} S_2(x_0, \Sigma)+\frac{|\mu_1|^2+|\mu_2|^2}{12} S_2(x_0, \Sigma) +\nonumber\\ &\,&+\ \frac{|\mu_1|^2-|\mu_2|^2}{4}(R_{1\bar{1}1\bar{1}}-R_{2\bar{2}2\bar{2}}), \label{eq:23}\\
\aint R(E',\bar{E}', Z,   \bar{Z})\, d\theta(Z)&\ge& \frac{1}{6} S_2(x_0,\Sigma). \label{eq:24}
\end{eqnarray}
Here $\mu_1, \mu_2$ are the singular values of the projection $P$ from $\Sigma'$ to $\Sigma$, and $\{E_1, E_2\}$ is a unitary basis of $\Sigma$ such that $Pv_1=\mu_1 E_1, Pv_2=\mu_2 E_2$.
\end{proposition}

The relevance with Theorem \ref{thm:11} is that at $x_0$ where $|s|^2$ attaints its maximum we have
\begin{eqnarray*}
0&\ge& \aint \partial_{v}\bar{\partial}_{\bar{v}} |s|^2 d\theta(v) =\aint \langle \nabla_v s, \bar{\nabla}_{\bar{v}}
\bar{s}\rangle + \sum_{i=1}^k |a_{ij}|^2
( R_{v\bar{v}i\,\overline{i}}+  R_{v\bar{v}j\, \overline{j}} )\, d\theta(v).
\end{eqnarray*}
The integral is clearly independent of the choice of a unitary frame of the two dimensional space spanned by $\{\frac{\partial}{\partial z_{i}}, \frac{\partial}{\partial z_{j}}\}$, or the choice of a unitary frame $\{E_1, E_2\}$ of $\Sigma$. If the right hand side of (\ref{eq:23}) had a positive lower bound, the maximum principle would  show that $|s|^2=0$ at $x_0$, thus $|s|^2=0$ everywhere, which gives another proof Theorem \ref{thm:11}.

Since the estimates of Proposition  \ref{prop:21} have other applications we include a proof here. The proof needs some basic algebra and
computations. Let $a\in \mathfrak{u}(m)$ be an element of the Lie algebra of $\mathsf{U}(m)$. Consider the function:
$$
f(t)=\aint H(e^{t a} X)\, d\theta(X).
$$
By the choice of $\Sigma$,  $f(t)$ attains its minimum at $t=0$. This implies that $f'(0)=0$ and $f''(0)\ge 0$. Hence
\begin{eqnarray}
&\, &\aint \left( R(a(X), \overline{X}, X, \overline{X}) +R(X, \bar{a}(\overline{X}), X, \overline{X})\right)\,
d\theta(X)=0; \label{eq:25}\\
&\,& \aint \left(R(a^2(X), \overline{X}, X, \overline{X}) +R(X, \bar{a}^2(\overline{X}), X, \overline{X})+4R(a(X),
\bar{a}(\overline{X}), X, \overline{X})\right) d\theta(X) \nonumber\\
&\, & + \aint \left( R(a(X), \overline{X}, a(X), \overline{X})+ R(X, \bar{a}(\overline{X}), X,
\bar{a}(\overline{X})\right) d\theta(X) \ge 0. \label{eq:26}
\end{eqnarray}
We exploit these by looking into some special cases of $a$. Let  $W\perp \Sigma$ and $Z\in \Sigma$ be two fixed vectors.
Let $a=\i \left(Z\otimes \overline{W}+W\otimes \overline{Z}\right)$. Then
$$
a(X)=\i \langle X, \overline{Z}\rangle W;\quad  a^2(X)=-\langle X, \overline{Z}\rangle Z.
$$

To show (\ref{eq:1st}), let us apply (\ref{eq:25})  to the above $a$ and also to the one with $W$ being replaced by $\i W$, and add the resulting two estimates together,   we get
$$ \aint \langle X, \overline{Z}\rangle R(W, \overline{X}, X, \overline{X}) \,d\theta(X) =0. $$
Take $Z=E_1$, we have
\begin{eqnarray*}
0 & = & \aint x_1 R(W, \overline{X}, X, \overline{X}) \,d\theta(X) \\
& = & \aint \big(|x_1|^4 R(W, \overline{E}_1, E_1, \overline{E_1}) + 2|x_1x_2|^2 R(W, \overline{E}_1, E_2, \overline{E_2}) \big) \,d\theta(X) \\
& = & \frac{1}{3} \big(R(W, \overline{E}_1, E_1, \overline{E_1}) +  R(W, \overline{E}_1, E_2, \overline{E_2}) \big) \\
& = & \frac{2}{3} \aint \big(|x_1|^2 R(W, \overline{E}_1, E_1, \overline{E_1}) + |x_2|^2 R(W, \overline{E}_1, E_2, \overline{E_2}) \big) \,d\theta(X) \\
& = & \frac{2}{3} \aint R(W, \overline{E}_1, X, \overline{X}) \,d\theta(X)
\end{eqnarray*}
Similarly, $\aint R(W, \overline{E}_2, X, \overline{X}) \,d\theta(X) = 0$, hence (\ref{eq:1st}) holds.

Next we prove (\ref{eq:24}). Applying (\ref{eq:26}) to the above special $a$ and also to the one with $W$ being replaced by $\sqrt{-1}W$, and add the resulting two estimates together,   we have that
\begin{equation}\label{eq:27}
4\aint |\langle X, \overline{Z}\rangle|^2R(W, \overline{W}, X, \overline{X}) d\theta(X)\ge \aint \langle X,
\overline Z\rangle  R(Z, \overline{X}, X, \overline{X}) +\langle Z, \overline{X}\rangle R(X, \overline{Z}, X,
\overline{X}).
\end{equation}
Let $Z=E_i$, we get
\begin{equation*}\label{eq:27}
4\aint |x_i|^2R(W, \overline{W}, X, \overline{X}) d\theta(X)\ge \aint x_i  R(E_i, \overline{X}, X, \overline{X}) +\overline{x}_i R(X, \overline{E}_i, X, \overline{X})\,d\theta .
\end{equation*}
Add up $i=1,2$, it yields
$$ 4\aint R(W, \overline{W}, X, \overline{X}) d\theta(X)\ge 2 \aint   R(X, \overline{X}, X, \overline{X}) \,d\theta = \frac{2}{3} S_2(x_0,\Sigma ),
$$
thus formula (\ref{eq:24}) holds.

To prove  (\ref{eq:23}) we need to
consider general $W$ which may not be perpendicular to $\Sigma$.
In other words, we consider the case $|Z|=|W|=1$ and $Z\in \Sigma$.
\begin{eqnarray*}
a(X)&=&\i\left( \langle X, \overline{Z}\rangle W+\langle X, \overline{W}\rangle Z\right) \\
a^2(X)&=& -\langle X, \overline{Z}\rangle \left(Z+\langle W, \overline{Z}\rangle W\right)-\langle X,
\overline{W}\rangle \left(W+\langle Z, \overline{W}\rangle Z\right).
\end{eqnarray*}
Apply this to (\ref{eq:26}) and also apply to $a$ with $W$ being replaced by $\i W$, add the results up we get the
estimate:
\begin{eqnarray}
&\,& 4\aint |\langle X, \overline{Z}\rangle|^2 R(W,\overline{W}, X, \overline{X})+ |\langle X, \overline{W}\rangle|^2
R(Z, \overline{Z},  X, \overline{X})d\theta(X) \nonumber \\
&\ge& \aint \langle X, \overline{Z}\rangle  R(Z, \overline{X}, X, \overline{X}) +\langle Z, \overline{X}\rangle R(X,
\overline{Z}, X, \overline{X})\, d\theta(X)\label{eq:28}\\
&\,& +\aint \langle X, \overline{W}\rangle  R(W, \overline{X}, X, \overline{X}) +\langle W, \overline{X}\rangle R(X,
\overline{W}, X, \overline{X})\, d\theta(X)\nonumber\\
&\,& +2\aint \langle X, \overline{Z}\rangle \langle X, \overline{W}\rangle R(W, \overline{X}, Z, \overline{X})+\langle
Z, \overline{X}\rangle \langle W, \overline{X}\rangle R(X, \overline{W}, X, \overline{Z})\, d\theta(X).\nonumber
\end{eqnarray}
Apply the above to $Z=E_i$ ($i=1, 2$) and sum the results together we have
\begin{eqnarray}
&\,& 4\aint  R(W,\overline{W}, X, \overline{X})+ |\langle X, \overline{W}\rangle|^2 \left(R_{1 \overline{1}  X
\overline{X}}  +R_{2 \overline{2}  X \overline{X}}\right) d\theta(X) \nonumber \\
&\ge& \frac{2}{3} S_2(x_0, \Sigma)+4\aint  \langle X, \overline{W}\rangle R(W, \overline{X}, X, \overline{X})+ \langle W,
\overline{X}\rangle R(X, \overline{W}, X, \overline{X})\, d\theta(X).\label{eq:29}
\end{eqnarray}
Now we want apply the above to all unit vectors $W
 \in \Sigma'$ and take the average. Denote by $P$ the orthogonal projection to $\Sigma$. Let $\{ v_1, v_2\}$ be a unitary basis of $\Sigma'$. Replacing $\{ v_1, v_2\}$ by a new unitary basis $\{ av_1+bv_2, -\overline{a}v_1 + \overline{b}v_2\}$ (where $|a|^2+|b|^2 =1$) if necessary, we pay assume that $Pv_1 \perp Pv_2$. So we can choose a unitary basis $\{ E_1, E_2\}$ of $\Sigma$ such that $v_1=\mu_1 E_1+\alpha E'$ and $v_2=\mu_2 E_2+ \beta E''$ with $\mu_i$ being the singular value of the projection to $\Sigma$ restricted to $\Sigma'$, and with $E', E'' \in \Sigma^{\perp}$. Now we  apply (\ref{eq:29}) to $W\in \mathbb{S}^3\subset \Sigma'$. First we observe that
\begin{eqnarray*}
2\aint  R(v_1,\overline{v}_1, X, \overline{X})+R(v_2,\overline{v}_2, X, \overline{X})\, d\theta(X)&=&
4\aint_{\mathbb{S}^3\subset \Sigma'} \aint R(W,\overline{W}, X, \overline{X})\, d\theta(X)\, d\theta(W).
\end{eqnarray*}
The second term on the left hand side of (\ref{eq:29}) has average value
 \begin{eqnarray*}
 L_2 &=&4\aint_{\mathbb{S}^3\subset \Sigma'} \aint|\langle X, \overline{W}\rangle|^2 \left(R_{1 \overline{1}  X
\overline{X}}  +R_{2 \overline{2}  X \overline{X}}\right) d\theta(X)\, d\theta(W)\\
&=& 2\aint\left(|\langle X, \overline{v}_1\rangle|^2 +|\langle X, \overline{v}_2\rangle|^2\right)\left(R_{1 \overline{1}  X
\overline{X}}  +R_{2 \overline{2}  X \overline{X}}\right) d\theta(X)
 \end{eqnarray*}
Express $X=x_1E_1+x_2 E_2$, we have
\begin{eqnarray*}
 2\aint |\langle X, \overline{v}_1\rangle|^2 \left(R_{1 \overline{1}  X
\overline{X}}  +R_{2 \overline{2}  X \overline{X}}\right) d\theta(X) &=& 2|\mu_1|^2\aint |x_1|^2( R_{1 \overline{1}  X
\overline{X}}  +R_{2 \overline{2}  X \overline{X}} ) \, d\theta\\
&=&2|\mu_1|^2\aint (|x_1|^4 R_{1\bar{1}1\bar{1}}+R_{1\bar{1}2\bar{2}}|x_1|^2 |x_2|^2)\, d\theta\\
 &\, & +2|\mu_1|^2\aint (|x_1|^4 R_{1\bar{1}2\bar{2}}+R_{2\bar{2}2\bar{2}}|x_1|^2|x_2|^2)\, d\theta\\
&=& \frac{2|\mu_1|^2}{3} R_{1\bar{1}1\bar{1}}+ |\mu_1|^2R_{1\bar{1}2\bar{2}}+\frac{|\mu_1|^2}{3}R_{2\bar{2}2\bar{2}}.
\end{eqnarray*}
Similarly we have
$$
 2\aint |\langle X, \overline{v}_2\rangle|^2 \left(R_{1 \overline{1}  X
\overline{X}}  +R_{2 \overline{2}  X \overline{X}}\right) d\theta(X)=\frac{2|\mu_2|^2}{3} R_{2\bar{2}2\bar{2}}+ |\mu_2|^2R_{1\bar{1}2\bar{2}}+\frac{|\mu_2|^2}{3}R_{1\bar{1}1\bar{1}}.
$$
In the mean time, the second term on the right hand side of (\ref{eq:29}) has average value
\begin{eqnarray*}
R_2 &=& 4\aint_{\mathbb{S}^3\subset \Sigma'}\aint    \langle X, \overline{W}\rangle R(W, \overline{X}, X, \overline{X})+ \langle W,
\overline{X}\rangle R(X, \overline{W}, X, \overline{X})\, d\theta(X)\, d\theta(W)
\\
&=& 2\aint    \langle X, \overline{v}_1\rangle R(v_1, \overline{X}, X, \overline{X})+ \langle v_1,
\overline{X}\rangle R(X, \overline{v}_1, X, \overline{X})\, d\theta(X)\\
&\,&+ 2\aint    \langle X, \overline{v}_2\rangle R(v_2, \overline{X}, X, \overline{X})+ \langle v_2,
\overline{X}\rangle R(X, \overline{v}_2, X, \overline{X})\, d\theta(X).
 \end{eqnarray*}
 We compute
\begin{eqnarray*}
2\aint    \langle X, \overline{v}_1\rangle R(v_1, \overline{X}, X, \overline{X})
 &=& 2 \aint x_1 (|\mu_1|^2R_{1\overline{X} X\overline{X}}+\overline{\mu}_1\alpha R_{E'\overline{X} X \overline{X}})d\theta\\
&=& 2|\mu_1|^2\aint x_1 R_{1\overline{X} X\overline{X}}\, d\theta+\frac{2}{3}\overline{\mu}_1\alpha (R_{E'\bar{1}1\bar{1}}+R_{E'\bar{1} 2\bar{2}})\\
&=&2|\mu_1|^2\aint \left(|x_1|^4 R_{1\bar{1}1\bar{1}}+2|x_1|^2|x_2|^2 R_{1\bar{1}2\bar{2}}\right)d\theta \\&=&\frac{2|\mu_1|^2}{3}(R_{1\bar{1}1\bar{1}}+R_{1\bar{1}2\bar{2}}).
\end{eqnarray*}
Hence after adding the result with its conjugation we have
$$
2\aint  \langle X, \overline{v}_1\rangle R(v_1, \overline{X}, X, \overline{X})+\langle v_1,
\overline{X}\rangle R(X, \overline{v}_1, X, \overline{X})\, d\theta(X) =\frac{4|\mu_1|^2}{3}(R_{1\bar{1}1\bar{1}}+R_{1\bar{1}2\bar{2}}).
$$
Similarly we also have
$$
2\aint  \langle X, \overline{v}_2\rangle R(v_2, \overline{X}, X, \overline{X})+\langle v_2,
\overline{X}\rangle R(X, \overline{v}_2, X, \overline{X})\, d\theta(X) =\frac{4|\mu_2|^2}{3}(R_{2\bar{2}2\bar{2}}+R_{1\bar{1}2\bar{2}}).
$$
Therefore we have
$$ R_2 = \frac{4|\mu_1|^2}{3}(R_{1\bar{1}1\bar{1}}+R_{1\bar{1}2\bar{2}}) + \frac{4|\mu_2|^2}{3}(R_{2\bar{2}2\bar{2}}+R_{1\bar{1}2\bar{2}}) .$$

Putting them all together and noting that $S_2(x_0, \Sigma)=R_{1\bar{1}1\bar{1}}+2R_{1\bar{1}2\bar{2}}+R_{2\bar{2}2\bar{2}}$ we get
\begin{eqnarray*}
2\aint  R(v_1,\overline{v}_1, X, \overline{X})+R(v_2,\overline{v}_2, X, \overline{X})\, d\theta(X)&\ge&
\frac{2}{3} S_2(x_0, \Sigma)+\frac{|\mu_1|^2+|\mu_2|^2}{6} S_2(x_0, \Sigma)\\
&\,&+ \ \frac{|\mu_1|^2-|\mu_2|^2}{2}(R_{1\bar{1}1\bar{1}}-R_{2\bar{2}2\bar{2}}).
\end{eqnarray*}
This proves (\ref{eq:23}), which completes the proof of the proposition.

\section{The high dimensional case}

Now for a $k$-dimensional subspace $\Sigma\subset T'_{x_0}M$ with $S_k(x_0, \Sigma)=\inf_{\Sigma'} S_k(x_0, \Sigma')$ we derive estimates similar to Proposition \ref{prop:21}.
\begin{proposition}\label{prop:31} Let $\Sigma, \Sigma'$ be two $k$-dimensional subspaces of $T_{x_0}'M$. Assume that $S_k(x_0, \Sigma)=\inf_{\Sigma'} S_k(x_0, \Sigma')$ and $\{v_1, \cdots, v_k\}$ and $\{ E_1, \ldots , E_k\}$ be  unitary frame at $x_0$ of $\Sigma'$ and $\Sigma$ respectively. Let $\{\mu_i\}$ be the singular values of the projection of $\Sigma'$ towards $\Sigma$.  Then for any $E\in \Sigma$, $E'\perp \Sigma$, we have
\begin{eqnarray}
\aint R(E,\overline{E}', Z, \overline{Z})d\theta(Z)&=&\aint R(E', \overline{E}, Z, \bar{Z})d\theta(Z)=0, \label{eq:3-1st} \\
\aint \left(  \sum_{j=1}^k R(v_j, \overline{v}_j, Z, \overline{Z})\right)\! d\theta(Z) &\ge& \frac{1}{k(k+1)}\left(\sum_{i=1}^k(1-|\mu_i|^2)\right) S_k(x_0, \Sigma)
\label{eq:31}\\
&\, &+\  \frac{1}{k} \sum_{i=1}^k \left(|\mu_i|^2 \sum_{j=1}^kR_{i\bar{i}j\bar{j}}\right),\nonumber\\
 \aint R(E',\overline{E}', Z,   \overline{Z})\, d\theta(Z)&\ge&\frac{S_k(x_0,\Sigma)}{k(k+1)}. \label{eq:32}
\end{eqnarray}
\end{proposition}
\begin{proof} Let $f(t)$ be the function constructed by the variation under the $1$-parameter family of unitary transformations.  The equations (\ref{eq:25}) and (\ref{eq:26}), as well as their proofs,  remain the same. The proof of (\ref{eq:3-1st}) and (\ref{eq:32}) are exactly analogous to that of (\ref{eq:1st}) and (\ref{eq:24}), so we omit it here.

 To prove (\ref{eq:31}) we apply (\ref{eq:28}) to $Z=E_i$ and add the results up we have
\begin{eqnarray}
&\,& 4\aint  R(W,\overline{W}, X, \overline{X})+ |\langle X, \overline{W}\rangle|^2 \left(\sum_{j=1}^kR_{j \overline{j}  X
\overline{X}} \right) d\theta(X) \label{eq:29-2} \\
&\ge& \frac{4}{k(k+1)} S_k(x_0, \Sigma)+(k+2)\aint  \langle X, \overline{W}\rangle R(W, \overline{X}, X, \overline{X})+ \langle W,
\overline{X}\rangle R(X, \overline{W}, X, \overline{X})\, d\theta(X).\nonumber
\end{eqnarray}
For the given $k$-planes $\Sigma$ and $\Sigma'$, we may always take unitary basis $\{ v_1, \ldots , v_k\}$ of $\Sigma'$ and unitary basis $\{ E_1, \ldots , E_k\}$ of $\Sigma$, so that the restriction on $\Sigma'$ of the projection map to $\Sigma$ is given by a diagonal matrix under these basis. That is,  $v_i=\mu_iE_i+\alpha_i E'_i$ for each $i$ with $E_i'\perp \Sigma$, and $\{\mu_i\}$ be the singular values of the projection from $\Sigma'$ to $\Sigma$.

Now we apply (\ref{eq:29-2}) to $W\in\mathbb{S}^{2k-1}\subset \Sigma'$ and take the average of the result.
We have that
\begin{eqnarray*}
\frac{4}{k}\aint \sum_{i=1}^k R(v_i,\overline{v}_i, X, \overline{X})\, d\theta(X)&=&
4\aint_{\mathbb{S}^{2k-1}\subset \Sigma'} \aint R(W,\overline{W}, X, \overline{X})\, d\theta(X)\, d\theta(W).
\end{eqnarray*}
Similarly we can calculate,
\begin{eqnarray*}
 &\,&4\aint_{\mathbb{S}^{2k-1}\subset \Sigma'} \aint|\langle X, \overline{W}\rangle|^2 \left(\sum_{j=1}^k R_{j \overline{j}  X
\overline{X}} \right) d\theta(X)\, d\theta(W)\\
&=& \frac{4}{k}\aint\left(\sum_{i=1}^k|\langle X, \overline{v}_i\rangle|^2\right)\left(\sum_{j=1}^k R_{j \overline{j}  X
\overline{X}}\right) d\theta(X)\\
&=& \frac{4}{k} \frac{1}{k(k+1)}\sum_{i=1}^k \left(|\mu_i|^2 \left(S_k+\sum_{j=1}^k R_{i\bar{i}j\bar{j}}\right)\right);
 \end{eqnarray*}
 while
\begin{eqnarray*}
&\,& (k+2)\aint_{\mathbb{S}^{2k-1}\subset \Sigma'}\aint    \langle X, \overline{W}\rangle R(W, \overline{X}, X, \overline{X})+ \langle W,
\overline{X}\rangle R(X, \overline{W}, X, \overline{X})\, d\theta(X)\, d\theta(W)
\\
&=& \frac{k+2}{k}\aint    \sum_{i=1}^k\langle X, \overline{v}_i\rangle R(v_i, \overline{X}, X, \overline{X})+\, \langle v_i,
\overline{X}\rangle R(X, \overline{v}_i, X, \overline{X})\, d\theta(X).
 \end{eqnarray*}
 Using (\ref{eq:3-1st}), the first half in the above can be further simplified into
 \begin{eqnarray*}
 &\,&\frac{k+2}{k}\aint    \sum_{i=1}^k\langle X, \overline{v}_i\rangle R(v_i, \overline{X}, X, \overline{X})\, d\theta(X)\\
&=&\quad \frac{k+2}{k}\aint    \sum_{i=1}^k x_i\left(|\mu_i|^2 R_{i \overline{X} X \overline{X}}+\overline{\mu}_i\alpha_i R_{E_i'\overline{X} X \overline{X}}\right)d\theta(X)\\
&=&\quad \frac{k+2}{k}\aint    \sum_{i=1}^k x_i\left(|\mu_i|^2 R_{i \overline{X} X \overline{X}}\right)d\theta(X)\\
&=&\quad \frac{k+2}{k}   \sum_{i=1}^k  |\mu_i|^2 \aint \left( |x_i|^4 R_{i\overline{i}i\overline{i}} + 2\sum_{j\neq i} |x_ix_j|^2 R_{i \overline{i} j \overline{j}}\right) d\theta(X)\\
&=& \quad \frac{k+2}{k}\frac{2}{k(k+1)} \sum_{i=1}^k \left(|\mu_i|^2 \sum_{j=1}^k R_{i\bar{i}j\bar{j}}\right).
 \end{eqnarray*}
 Putting the above together we have (\ref{eq:32}).
\end{proof}

\vspace{0.05cm}

\section*{Acknowledgments} { We would like to thank James McKernan for his interests and discussions, Masataka Iwai for sharing his example, Man-Chun Lee and Luen-Fai Tam for pointing out a discrepancy in the earlier version of the paper. }


\begin{thebibliography}{A}
	


\bibitem{Andrews} B. Andrews, \textit{
Noncollapsing in mean-convex mean curvature flow.} (English summary)
Geom. Topol. \textbf{16} (2012), no. 3, 1413--1418. 	

\bibitem{Andrews-C} B. Andrews and J.  Clutterbuck, \textit{ Proof of the fundamental gap conjecture.} J. Amer. Math. Soc. \textbf{24}(2011), no. 3, 899--916


		\bibitem{HeierWong}
G. Heier and B. Wong, \textit{ On projective K\"ahler manifolds of partially positive curvature and rational connectedness.} ArXiv:1509.02149.



\bibitem{Hitchin}
N. Hitchin, \textit{On the curvature of rational surfaces.} In Differential Geometry ({\em Proc. Sympos. Pure Math., Vol
XXVII, Part 2, Stanford University, Stanford, Calif., 1973}), pages 65-80. Amer. Math. Soc., Providence, RI, 1975.

\bibitem{HJ} R. Horn and C. Johnson, \textit{ Matrix analysis.} 2nd Edition. Cambridge University Press, 2013.




\bibitem{Iwai} M. Iwai, \textit{Private communication.} April 28, 2018.



\bibitem{KMM} J. Koll\'ar, Y. Miyaoka, and S. Mori,  \textit{Rationally connected varieties.} J. Alg. Geom. \textbf{1} (1992), 429--448.



 \bibitem{Ko2} S. Kobayashi,    \textit{Differential geometry of complex vector bundles.} Publications of the
     Mathematical Society of Japan, 15. Kano Memorial Lectures, 5. Princeton University Press, Princeton, NJ;
     Princeton University Press, Princeton, NJ, 1987. xii+305 pp.


\bibitem{Liu}G. Liu, \textit{Three-circle theorem and dimension estimate for holomorphic functions on K\"ahler
    manifolds.} Duke Math. J. \textbf{165} (2016), no. 15, 2899--2919.



 \bibitem{MK}  J.   Morrow and K.  Kodaira, \textit{ Complex manifolds.} Holt. Rinehart and Winston, New
     York-Montreal-London, 1971.

\bibitem{Ni-JDG} L. Ni,  \textit{ Vanishing theorems on complete K\"ahler manifolds and their applications.} J.
    Differential Geom. \textbf{50} (1998), no. 1, 89--122.

\bibitem{Ni} L. Ni, \textit{ Estimates on the modulus of expansion for vector fields solving nonlinear equations.}  J. Math. Pures Appl. (9)\textbf{99}:1 (2013), 1--16.

\bibitem{Ni-ijm} L. Ni, \textit{ General Schwarz lemmata and their applications. } Internat. J. Math. \textbf{30}(2019), no. 13, 1940007, 17 pp.

\bibitem{Ni-19} L. Ni, \textit{The fundamental group, rational connectedness and the positivity of K\"ahler manifolds,} arXiv:1902.00974. \footnote{The most updated version is available at http://math.ucsd.edu/\~\, lni/academic/Pi1ORicci8.pdf.}

\bibitem{NZ} L. Ni and F.-Y Zheng, \textit{Comparison and vanishing  theorems for K\"ahler manifolds.} Calc. Var. Partial Differential Equations, \textbf{57}(2018), no. 6, Art. 151, 31 pp.

\bibitem{NZ-19} L. Ni and F.-Y Zheng, \textit{On orthogonal Ricci curvature.} Advances in complex geometry, 203--215, Contemp. Math., \textbf{735}, Amer. Math. Soc., Providence, RI, 2019.








 \bibitem{Tsu}  Y.   Tsukamoto,\textit{ On K\"ahlerian manifolds with positive holomorphic sectional curvature.}      Proc. Japan Acad. \textbf{33} (1957), 333--335.


\bibitem{Voi} C.  Voisin, \textit{ On the homotopy types of compact K\"ahler and complex projective manifolds. } Invent. Math. \textbf{157} (2004), no. 2, 329--343.

     \bibitem{Wilking} B. Wilking,  \textit{A Lie algebraic approach to Ricci flow invariant curvature condition and Harnack inequalities.}  J. reine angew. Math. (Crelle), \textbf{679} (2013), 223--247.
    		

\bibitem{WuYau} D, Wu and S.-T. Yau, \textit{Negative holomorphic curvature and positive canonical bundle.}  	
Invent. Math. \textbf{204}(2016), 595--604.



\bibitem{XYang}
X. Yang, \textit{ RC-positivity, rational connectedness, and Yau's conjecture.} Cambridge J. Math. \textbf{6}(2018), 183--212.


\bibitem{YZ-Hpositive}
B. Yang and F. Zheng, \textit{ Hirzebruch manifolds and  positive holomorphic sectional curvature.} Ann. Inst. Fourier (Grenoble), \textbf{69} (2019), no. 6, 2589--2634.
		


 \bibitem{YauP} S.-T. Yau, \textit{ Problem section. } Seminar on Differential Geometry, pp. 669--706, Ann. of Math.
     Stud., 102, Princeton Univ. Press, Princeton, N.J., 1982.

\end{thebibliography}
\end{document}